\documentclass[11pt,reqno]{amsart}
\usepackage{graphicx, enumerate, url}
\usepackage[margin=1in]{geometry}
\usepackage{amssymb,mathrsfs,mathtools}
\usepackage{color}
\newtheorem{theorem}{Theorem}[section]

\newtheorem{lemma}[theorem]{Lemma}
\newtheorem{corollary}[theorem]{Corollary}

\theoremstyle{definition}
\newtheorem{definition}[theorem]{Definition}

\theoremstyle{remark}

\newcommand{\red}[1]{#1}
\newcommand{\blue}[1]{#1}

\begin{document}
\title{The Computational Complexity of Duality}
\author[S.~Friedland]{Shmuel~Friedland}
\address{Department of Mathematics, Statistics and Computer Science,  University of Illinois, Chicago}
\email{friedlan@uic.edu}
\author[L.-H.~Lim]{Lek-Heng~Lim}
\address{Computational and Applied Mathematics Initiative, Department of Statistics,
University of Chicago}
\email{lekheng@galton.uchicago.edu}
\begin{abstract}
We show that for any given norm ball or proper cone, weak membership in its dual ball or dual cone  is polynomial-time reducible to  weak membership in the given ball or cone. A consequence is that the weak membership or membership problem for a ball or cone is NP-hard if and only if the corresponding problem for the dual ball or cone is NP-hard. In a similar vein, we show that computation of the dual norm of a given norm is polynomial-time reducible to computation of the given norm. This extends to convex functions satisfying a polynomial growth condition: for such a given function, computation of its Fenchel dual/conjugate is polynomial-time reducible to computation of the given function. Hence the computation of a norm or a convex function of polynomial-growth is NP-hard if and only if the computation of its dual norm or Fenchel dual is NP-hard. We discuss implications of these results on the weak membership problem for a symmetric convex body and its polar dual, the polynomial approximability of Mahler volume, and the weak membership problem for the epigraph of a convex function with polynomial growth and that of its Fenchel dual.
\end{abstract}

\keywords{dual norm, dual cone, Fenchel dual, NP-hard, weak membership, approximation}

\subjclass[2010]{15B48, 52A41, 65F35, 90C46, 90C60}

\maketitle

\section{Introduction}

In convex optimization, we often encounter problems that involve one of the following notions of duality. For convex sets: (i) norm balls and their polar duals, (ii) proper cones and their dual cones; for convex functions:
(iii) norms and their dual norms; (iv) functions and their Fenchel duals. The main goal of this article is to establish the equivalence between the polynomial-time computability or NP-hardness of these objects and their duals.

We will first show in Section~\ref{sec:norm}  that the weak membership problem for a norm ball is NP-hard (resp.\ is polynomial-time) if and only if the weak membership problem for its dual norm ball is NP-hard (resp.\ is polynomial-time). For readers unfamiliar with the notion, NP-hardness of \emph{weak} membership is a \emph{stronger} statement than NP-hardness of membership, i.e., the latter is implied by the former. Since every symmetric convex compact set with nonempty interior is a norm ball, the result applies to such objects and their polar duals as well.

In Section~\ref{sec:approx} we show that the approximation of a norm to arbitrary precision is NP-hard (resp.\ is polynomial-time) if and only if  weak membership  in the unit ball of the norm is NP-hard (resp.\ is \blue{polynomial-time}).  A consequence is that if the weak membership problem for a norm ball is polynomial-time decidable, then its Mahler volume is polynomial-time approximable. In fact, computation of Mahler volume is polynomial-time reducible to the weak membership problem for a norm ball.

In Section~\ref{sec:cones}, we establish an analogue of our norm ball result for proper cones, showing that  the weak membership problem for such a cone  is NP-hard (resp.\ is polynomial-time) if and only if the weak membership problem for its dual cone can be decided is NP-hard (resp.\ is polynomial-time).

We conclude by showing in Section~\ref{sec:fenchel} that for convex functions that satisfy a polynomial-growth condition, its Fenchel dual must also satisfy the same condition with possibly different constants. A consequence of this is that such a function is polynomial-time approximable to arbitrary precision if and only if its Fenchel dual is also  polynomial-time approximable to arbitrary precision. On the other hand, such a function is NP-hard to approximate if and only if its Fenchel dual is NP-hard to approximate.

\section{Weak membership, weak validity, and polynomial-time reducibility}\label{sec:weak}

We introduce some basic terminologies based on \cite[Chapter~2]{GLS88}) with some natural extensions for our context. 
Let $B(x,\delta)$ denote the closed Euclidean norm ball of radius $\delta > 0$ centered at $x$ in $\mathbb{R}^n$.
For any $\delta>0$ and any $K \subseteq \mathbb{R}^n$,  we define respectively a `thickened' $K$ and a `shrunkened' $K$ by
\begin{equation}\label{eq:thick-shrunk}
S(K,\delta)\coloneqq \bigcup\nolimits_{x\in K} B(x,\delta)\quad \text{and}\quad S(K,-\delta)\coloneqq \{x\in K: B(x,\delta)\subseteq K\}.
\end{equation}
Note that if $K$ has no interior point, then $S(K,-\delta) = \varnothing$.

\begin{definition}\label{def:mem}
Let $K \subseteq \mathbb{R}^n$ be a convex set with nonempty interior.
\begin{enumerate}[\upshape (i)]
\item The \emph{membership problem} (\textsc{mem}) for $K$ is: Given $x\in\mathbb{Q}^n$, determine if $x$ is in $K$.

\item The \emph{weak membership problem} (\textsc{wmem}) for $K$ is: Given $x\in\mathbb{Q}^n$ and a rational $\delta>0$,  assert that $x\in S(K,\delta)$ or $x\notin S(K,-\delta)$. 

\item\red{The \emph{weak violation problem} (\textsc{wviol}) problem for $K$ is: Given $c\in\mathbb{Q}^n$ and rational $\gamma,\varepsilon>0$,
either assert that $c^\mathsf{T} x\le \gamma +\varepsilon$ for all $x\in S(K,-\varepsilon)$, or find $y\in S(K,\varepsilon)$ with $c^\mathsf{T} y \ge \gamma -\varepsilon$.}

\item The \emph{weak validity problem} (\textsc{wval}) problem for $K$ is: Given $c\in\mathbb{Q}^n$ and rational $\gamma,\varepsilon>0$,
either assert that $c^\mathsf{T} x\le \gamma +\varepsilon$ for all $x\in S(K,-\varepsilon)$, or assert that $c^\mathsf{T} x \ge \gamma -\varepsilon$ for some $x\in S(K,\varepsilon)$.

\item The \emph{weak optimization problem} (\textsc{wopt}) problem for $K$ is: Given $c\in\mathbb{Q}^n$ and a rational $\varepsilon>0$,
either find $y \in\mathbb{Q}^n$ such that  $y\in S(K,\varepsilon)$ and $c^\mathsf{T} x\le c^\mathsf{T} y +\varepsilon$ for all $x\in S(K,-\varepsilon)$, or assert that $S(K,\varepsilon) = \varnothing$.
\end{enumerate}
\end{definition}

For the benefit of readers unfamiliar with these notions, we highlight that in our weak membership problem, there are  $x$'s that satisfy both $x\in S(K,\delta)$ and $x\notin S(K,-\delta)$ simultaneously. So if we can ascertain \textsc{mem}, we can ascertain \textsc{wmem}, but not conversely. A consequence is that if \textsc{wmem} problem for $K$ is NP-hard, then \textsc{mem} for $K$ is also NP-hard.

There will be occasions, particularly in Section~\ref{sec:cones}, when we have to discuss weak membership and weak validity of a convex set $K \subseteq \mathbb{R}^n$ of  positive codimension, i.e., contained in an affine subspace of dimension \blue{less than} $n$. As a subset of $\mathbb{R}^n$, $K$ will have no interior points and the \textsc{wmem} and \textsc{wval} as defined above would make little sense \blue{as $S(K,-\delta)=\varnothing$.}  With this in mind, we introduce the following variant of Definition~\ref{def:mem} that makes use of the interior of $K$ relative to $H$, an affine subspace of minimal dimension that contains $K$, i.e., $H$ is the affine hull of $K$.  We start by defining
\[
S_H(K,-\delta)\coloneqq \{x\in K: B(x,\delta)\cap H\subseteq K\}\quad \blue{\text{and}\quad S_H(K,\delta)\coloneqq S(K,\delta)\cap H}.
\]
Note that if $K \ne \varnothing$, then there exists $\varepsilon >0$ such that $S_H(K,-\delta) \ne \varnothing$ for each $\delta \in (0,\varepsilon)$, even if $K$ has no interior point. If $K$ has nonempty interior, then $H = \mathbb{R}^n$ and $S_H(K,-\delta) =S(K,-\delta)$.

\begin{definition}\label{def:relmem}
Let $K \subseteq \mathbb{R}^n$ be a convex set and let $H = \operatorname{aff}(K)$ be its affine hull.
\begin{enumerate}[\upshape (i)]
\item The weak membership problem (\textsc{wmem}) for $K$ relative to $H$ is: Given $x\in\mathbb{Q}^n$ and a rational number $\delta>0$,  assert that $x\in \blue{S_H(K,\delta)} $ or  $x\notin S_H(K,-\delta)$.  

\item The weak validity problem (\textsc{wval}) problem for $K$ relative to $H$ is: Given $c\in\mathbb{Q}^n$ and rational numbers $\gamma,\varepsilon>0$,
either assert that $c^\mathsf{T} x\le \gamma +\varepsilon$ for all $x\in S_H(K,-\varepsilon)$, or assert that $c^\mathsf{T} x \ge \gamma -\varepsilon$ for some $x\in \blue{S_H(K,\varepsilon)}$.
\end{enumerate}
\end{definition}

An implicit assumption throughout this article is that when we study the computational complexity of \textsc{wmem} and \textsc{wval} problems for a convex set $K \subseteq \mathbb{R}^n$ with nonempty interior, we assume that we know a point $a \in \mathbb{Q}^n$ and a rational $r > 0$ such that the Euclidean norm ball $B(a,r) \subseteq K$. This mild \red{\emph{centering assumption}} guarantees that $K$ is `centered' in the sense of \cite[Definition~2.1.16]{GLS88} and is needed whenever we invoke Yudin--Nemirovski Theorem \red{\cite{YN76}} and \cite[Theorem~4.3.2]{GLS88}.

Recall that a problem $\mathscr{P}$ is said to be \emph{polynomial-time reducible}  \cite[p.~28]{GLS88} to a problem $\mathscr{Q}$ if there is a polynomial-time algorithm $A_\mathscr{P}$ for solving $\mathscr{P}$ by making a polynomial number of oracle calls to an algorithm $A_\mathscr{Q}$  for solving $\mathscr{Q}$. This notion of polynomial-time reducibility is also called Cook  or Turing reducibility and will be the one used throughout our article. There is also a more restrictive notion of polynomial-time reducibility that allows only a single oracle call to  $A_\mathscr{Q}$ called Karp or many-one reducibility.

Note that if $A_\mathscr{Q}$ is a polynomial-time algorithm for $\mathscr{Q}$, then $A_\mathscr{P}$ is a polynomial-time algorithm for $\mathscr{P}$. Consequently, if $\mathscr{Q}$ is computable in polynomial-time, then so is $\mathscr{P}$. On the other hand, if $\mathscr{P}$ is NP-hard, then so is $\mathscr{Q}$.

We say that $\mathscr{P}$ and $\mathscr{Q}$ are \emph{polynomial-time inter-reducible} if $\mathscr{P}$ is polynomial-time reducible to $\mathscr{Q}$ and $\mathscr{Q}$ is polynomial-time reducible to $\mathscr{P}$.  The polynomial-time inter-reducibility of two problems $\mathscr{P}$ and $\mathscr{Q}$ implies that they are in the same time-complexity class\footnote{Assuming that the complexity class is defined by polynomial-time inter-reducibility.} whatever it may be. Nevertheless, in this article we will restrict ourselves to just polynomial-time computability and NP-hardness, the two most \red{often} used cases in optimization.

\section{Weak membership in dual norm balls}\label{sec:norm}

Our \blue{technique} for this section relies on tools introduced in \cite[Chapter 4]{GLS88} and is inspired by \cite[Section~6.1]{Gu02}. While our discussion below is over $\mathbb{R}$, it is easy to extend it to $\mathbb{C}$ since $\mathbb{C}^{n}$ \red{may be} identified with $\mathbb{R}^{2n}\equiv \mathbb{R}^{n}\times \mathbb{R}^{n}$, where $z = x+\sqrt{-1}y\in \mathbb{C}^n$ is identified with $(x,y)\in\mathbb{R}^{n}\times \mathbb{R}^{n}$.  
A norm $\nu:\mathbb{C}^n\to [0,\infty)$ induces a norm $\tilde \nu:\mathbb{R}^{2n}\to [0,\infty)$ via
$\tilde \nu\bigl((x,y)\bigr)\coloneqq \nu(x+\sqrt{-1}y)$ and we may identify $\nu$ with $\tilde \nu$.  In particular, the Hermitian norm on $\mathbb{C}^n$ 
gives exactly the Euclidean norm on $\mathbb{R}^{2n}$.  Hence for the purpose of this article, it suffices to consider norms over real vector spaces.

Let $\nu:\mathbb{R}^n\to [0,\infty)$ be a norm and denote the closed ball and open ball centered at $a \in \mathbb{R}^n$ of radius $r > 0$ with respect to  the norm $\nu$ by
\[
B_\nu(a,r)\coloneqq \{x\in \mathbb{R}^n : \nu(x-a)\le r\}\quad\text{and}\quad
B^\circ_\nu(a,r)\coloneqq \{x\in \mathbb{R}^n : \nu(x-a) < r\}
\]
respectively. For the special case $a = 0$ and $r =1$, we write $B_\nu\coloneqq B_\nu(0,1)$ and  $B^\circ_\nu\coloneqq B^\circ_\nu(0,1)$  for the closed  and open unit balls.  
For the special case $\nu = \|\cdot\|$, the Euclidean norm on $\mathbb{R}^n$, we write $B(a,r)\coloneqq B_{\|\cdot \|}(a,r)$ and  $B^\circ(a,r)\coloneqq B^\circ_{\|\cdot \|}(a,r)$, dropping the subscript.
Since all norms on $\mathbb{R}^n$ are equivalent, it follows that there exist 
constants $K_\nu\ge k_\nu>0$ such that 
\begin{equation}\label{mnuequiv}
k_\nu\|x\|\le \nu(x)\le K_\nu\|x\|\quad \blue{\text{for all}\; x\in \mathbb{R}^n.}
\end{equation}
There is no loss of generality in assuming that $k_\nu$ and $K_\nu$ are rational\footnote{If not just pick a smaller $k_\nu$ or a larger $K_\nu$ that is rational.} and we may denote the number of bits required to specify them by $\langle k_\nu \rangle$ and $ \langle K_\nu \rangle$ respectively.

Recall that the \emph{dual norm} of $\nu$, denoted $\nu^*$, is given by 
\[
\nu^*(x)=\max\{\lvert y^\mathsf{T} x \rvert : \nu(y)\le 1\}
\]
for every $x\in\mathbb{R}^{n}$. Hence
\begin{equation}\label{constdualnrm}
 \frac{1}{K_\nu}\|x\|\le \nu^*(x)\le \frac{1}{k_\nu}\|x\|\quad\blue{\text{for all}\; x\in \mathbb{R}^n.}
\end{equation}

Observe first that  $B(0,1/K_\nu)\subseteq B_\nu\subseteq B(0,1/k_\nu)$ \red{and $B(0,k_\nu)\subseteq B_{\nu^*}\subseteq B(0,K_\nu)$.  So $B_\nu$ and $B_{\nu^*}$ satisfy the centering assumption after Definition~\ref{def:relmem} with $a=0$.} Hence
\[
\langle B_\nu \rangle\coloneqq \langle n \rangle+\langle k_\nu \rangle+\langle K_\nu \rangle
\]
may be regarded as the encoding length of $B_\nu$ in number of bits.  \red{A norm or unit-norm ball may therefore be encoded (for a Turing machine) in finitely many bits as  $(n, k_\nu, K_\nu) \in \mathbb{Q}^3$.  Whenever we discuss the computation of a norm, we implicitly assume knowledge of $(n, k_\nu, K_\nu)$, i.e., an algorithm would have access to their values.}

The main result of this section is the polynomial-time inter-reducibility between a norm and its dual.
\begin{theorem}\label{polduality}  
Let $\nu$ be a norm and $\nu^*$ be its dual norm. The \textsc{wmem} problem for the unit ball of $\nu^*$  is polynomial-time reducible to
the \textsc{wmem} problem for the unit ball of $\nu$.
\end{theorem}

We will prove this result via two intermediate lemmas. 
A key step in our proof depends on the Yudin--Nemirovski Theorem \red{\cite{YN76}}, which may be stated as follows \cite[Theorem~4.3.2]{GLS88}.
\begin{theorem}[Yudin--Nemirovski]\label{thm:YN}
The \textsc{wval} problem for $B_\nu$ is polynomial-time reducible to the \textsc{wmem} problem for $B_{\nu}$. More generally this holds for any convex set with nonempty interior $K \subseteq \mathbb{R}^n$ for which we have knowledge of $a \in \mathbb{Q}^n$ and $0< r \le R \in \mathbb{Q}$ such that $B(a,r) \subseteq K \subseteq B(0,R)$.
\end{theorem}
The original Yudin--Nemirovski Theorem is in fact stronger than the version stated here, allowing the weak violation problem \textsc{wviol} to be reduced to \textsc{wmem}. Nevertheless in this article we will only require the weaker result with \textsc{wval} in place of \textsc{wviol}.

For a compact set $K\subset \mathbb{R}^n$ and $c\in \mathbb{R}^n$, \blue{the \emph{support function} of $K$ at $c$ is}
\[
\max(K,c)\coloneqq \max \{c^\mathsf{T} x  :  x \in K\}.
\]
In particular, observe that 
\[
\nu(x)=\max(B_{\nu^*},x).
\]
\begin{lemma}\label{auxineq}  Let $\nu$ be a norm on $\mathbb{R}^n$ and $\delta>0$.  Then we have inclusions
\begin{gather}
(1+k_\nu\delta)B_\nu\subseteq S(B_\nu,\delta)\subseteq (1+K_\nu\delta)B_\nu,\label{auxineq3}
\\
(1-K_\nu\delta)B_\nu\subseteq S(B_\nu,-\delta)\subseteq (1-k_\nu\delta)B_\nu,\label{auxineq4}
\end{gather}
whenever  $K_\nu\delta< 1$, and the inequalities
\begin{gather}
\label{auxineq1} \left(1-\frac{\delta}{k_\nu}\right)\nu(x) \le \max \bigl(S(B_{\nu^*},-\delta),x\bigr)\le 
\left(1-\frac{\delta}{K_\nu}\right)\nu(x),\\
\label{auxineq2}
\left(1+\frac{\delta}{K_\nu}\right)\nu(x)\le \max \bigl(S(B_{\nu^*},\delta),x\bigr)\le \left(1+\frac{\delta}{k_\nu}\right)\nu(x),
\end{gather}
whenever $\delta/k_\nu< 1$.
\end{lemma}
\begin{proof}  To prove \eqref{auxineq3}, observe that
\[
k_\nu B_\nu\subseteq B(0,1)\subseteq K_\nu B_\nu,\qquad  k_\nu B^\circ_\nu\subseteq B^\circ(0,1)\subseteq K_\nu B^\circ_\nu,
\]
and thus
\[
B_\nu(x,k_\nu \delta) \subseteq B(x,\delta) \subseteq B_\nu(x,K_\nu \delta),\qquad
B^\circ_\nu(x,k_\nu \delta) \subseteq B^\circ(x,\delta) \subseteq B^\circ_\nu(x,K_\nu \delta).
\]
Also, $ \bigcup_{x\in B_\nu} B_\nu(x, r) =B_\nu(0,1+r)$ by the defining properties of a norm. Hence
\[
S(B_\nu,\delta) =\bigcup_{x\in B_\nu} B(x,\delta) \subseteq \bigcup_{x\in B_\nu} B_\nu(x, K_\nu\delta) =B_\nu(0,1+K_\nu\delta).
\]
On the other  hand,
\[
S(B_\nu,\delta) =\bigcup_{x\in B_\nu} B(x,\delta) \supseteq \bigcup_{x\in B_\nu} B_\nu(x, k_\nu\delta)  =B_\nu(0, 1+k_\nu\delta).
\]

To prove  \eqref{auxineq4}, let $T=\bigcup_{x\, : \,\nu(x)=1} B^\circ(x,\delta)$ and so $S(B_\nu,-\delta)=B_\nu\setminus T$.
Let 
\[T_1=\bigcup_{x\, : \,\nu(x)=1}B^\circ_\nu(x, K_\nu\delta), \qquad T_2=\bigcup_{x\, : \, \nu(x)=1} 
B^\circ_\nu(x, k_\nu\delta).
\]
Since $T_1\supseteq T$ and $T_2\subseteq T$, we obtain
\[
S(B_\nu,-\delta)\supseteq B_\nu\setminus T_1 = (1-K_\nu\delta)B_\nu,
\qquad
S(B_\nu,-\delta)\subseteq B_\nu\setminus T_2 =(1-k_\nu\delta)B_\nu.
\]
The last two inequalities follow from the first two inclusions and \eqref{constdualnrm}.
\end{proof}

\begin{lemma}\label{wmemwval} Let $k_\nu\ge 2$.  Then the solution to \textsc{wval} problem for $B_{\nu^*}$ gives the solution to \textsc{wmem} problem for $B_\nu$.
\end{lemma}
\begin{proof} Let $x\in\mathbb{Q}^n$ and  $\delta\in (0, \frac{1}{2})\cap \mathbb{Q}$.  We choose $\gamma=1$.
Suppose that $x^\mathsf{T} y\le 1+ \delta$ for all $y\in S(B_{\nu^*},-\delta)$.  Then $\max\bigl(S(B_{\nu^*},-\delta),x\bigr)\le 1+\delta$  and by \eqref{auxineq1} we have
\[
\nu(x)\le \frac{1+\delta}{1-\delta/k_\nu}.
\]
Since $k_\nu\ge 2$, it follows  that
\[
\frac{1+\delta}{1-\delta/k_\nu}\le
1+k_\nu\delta.
\]
It follows from \eqref{auxineq3} that $x\in S(B_\nu,\delta)$.

Suppose that $x^\mathsf{T} y> 1- \delta$ for some $y\in S(B_{\nu^*},\delta)$. Then $\max\bigl(\red{S(B_{\nu^*},\delta),x}\bigr)>1-\delta$ and we deduce from \eqref{auxineq2}  that
\[
\nu(x)> \frac{1-\delta}{1+ \delta/k_\nu}.
\]
As straightforward calculation shows that
\[
\frac{1-\delta}{1+ \delta/k_\nu }\ge 1-k_\nu\delta.
\]
It follows from \eqref{auxineq4} that $x\notin S(B_\nu,-\delta)$.   
\end{proof} 

\begin{proof}[of Theorem~\ref{polduality}]
We observe that the assumption $k_\nu\ge 2$ in Lemma~\ref{wmemwval} is not restrictive. Let $r\ge 2/k_\nu$.  Then a new norm defined by $ \nu_r(x)=r\nu(x)$ would satisfy the assumption.
Now note that $x\in B_\nu$ if and only if $\frac{1}{r}x\in B_{\nu_r}$.  With this observation,
Theorem~\ref{polduality} follows from
\[
\text{\textsc{wmem} for }\red{B_{\nu^*}}  \Rightarrow  
\text{\textsc{wval} for }\red{B_{\nu^*}}  \Rightarrow  
\text{\textsc{wmem} for }\red{B_{\nu}}  \Rightarrow  
\text{\textsc{wval} for }\red{B_{\nu}}  \Rightarrow 
\text{\textsc{wmem} for }\red{B_{\nu^*}}.
\]
Here $\mathscr{P} \Rightarrow \mathscr{Q}$ means that $\mathscr{Q}$ is polynomial-time reducible to $\mathscr{P}$. Yudin--Nemirovski Theorem gives the first and third reductions whereas Lemma~\ref{wmemwval} gives the second and last reductions.
\end{proof}

Since taking \blue{the} dual of a dual norm gives us back the original norm, we have the following corollary.
\begin{corollary}
The \textsc{wmem} problem for the unit ball of a norm $\nu$  is polynomial-time decidable (resp.\ NP-hard)
if and only if the \textsc{wmem} problem for the unit ball of the dual norm $\nu^*$ is polynomial-time decidable (resp.\ NP-hard).
\end{corollary}

Since every centrally symmetric compact convex set with nonempty interior is a norm ball for some norm and its \emph{polar dual} is exactly the norm ball for the corresponding dual norm, we immediately have the following.
\begin{corollary}\label{cor:cscc}
Let $C$ be a  centrally symmetric compact convex set with nonempty interior in $\mathbb{R}^n$ and
\[
C^* = \{ x \in \mathbb{R}^n : x^\mathsf{T} y \le 1 \}
\]
be its polar dual. Then \textsc{wmem} in $C$ is polynomial-time inter-reducible to the \textsc{wmem} in $C^*$. In particular, if one is polynomial-time decidable  (resp.\ NP-hard), then so is the other.
\end{corollary}
 
\section{Approximation of dual norms}\label{sec:approx}

In this section we  show that for a given norm $\nu:\mathbb{R}^n \to [0,\infty)$ satisfying 
\eqref{mnuequiv} for $k_\nu, K_\nu \in \mathbb{Q}$, \textsc{wmem} in $B_\nu$ with respect to  $\delta \in \mathbb{Q}$ is polynomial-time inter-reducible with a $\delta$-approximation of the norm $\nu$.

\begin{definition}\label{def:approx}
Let  $\nu:\mathbb{R}^n \to [0,\infty)$ be a norm satisfying 
\eqref{mnuequiv} for $k_\nu, K_\nu \in \mathbb{Q}$.
The \emph{approximation problem} (\textsc{approx}) for $\nu$ is:
Let $\delta \in \mathbb{Q}$ and $\delta > 0$.  Given any $x\in \mathbb{Q}^n$ with $1/2 < \|x\| < 3/2$, compute an approximation $\omega(x) \in \mathbb{Q}$ such that
\begin{equation}\label{approxnu}
\blue{\nu(x) -\delta<\omega(x)<\nu(x) +\delta}.
\end{equation}
We call $\omega$ a \emph{$\delta$-approximation} of $\nu$.
\end{definition}
\red{The annulus $1/2 < \|x\| < 3/2$, where $x$ has rational coordinates, is intended as a rational approximation of the unit sphere $\|x\|=1$ in $\mathbb{R}^n$ --- points on the unit sphere that do not have rational coordinates can be approximated by rational points in the annulus}
The requirement that $1/2 < \|x\| < 3/2$ is not restrictive since we may always scale any given $x$ to meet this condition in polynomial-time. Note that an approximation problem has $n+\langle \delta \rangle+\langle K_\nu \rangle+\langle k_\nu \rangle $ input bits. If we say that such  a problem can be solved in polynomial time, we mean time polynomial in this number of input bits.

\begin{theorem}\label{nuapproxwm}
Let  $\nu:\mathbb{R}^n \to [0,\infty)$ be a norm satisfying 
\eqref{mnuequiv} for $k_\nu, K_\nu \in \mathbb{Q}$. Then the following problems are polynomial-time inter-reducible:
\begin{enumerate}[\upshape (i)]
\item The approximation problem for $\nu$.
\item The weak membership problem for $B_\nu$.
\end{enumerate}

\end{theorem}
\begin{proof}  Let us use (i) as an oracle and solve (ii).  Let $x\in\mathbb{Q}^n$ and a rational $\delta>0$ be given.
If $\|x\|\le 1/K_\nu$, then $\nu(x)\le 1$, and so $x\in S(B_\nu,\delta)$.
If $\|x\|\ge 1/k_\nu$, then $\nu(x)\ge 1$, and so $x\notin S(B_\nu,-\delta)$. 

It remains to check the case $\|x\|\in (1/K_\nu, 1/k_\nu)$. Let $r\in (2\|x\|/3, 2\|x\|)\cap \mathbb{Q}$ and let $y\coloneqq  x/ r$. Observe that $\nu(y)\in (k_\nu/2,3K_\nu/2)$. Now let  $\varepsilon=k_\nu^2\delta/4$ and \blue{let} $\omega(y)$ be an $\varepsilon$-approximation of $\nu(y)$. Assume first that 
\[
r\omega(y)\le 1+k_\nu\delta -\frac{2 \varepsilon}{k_\nu}=1+\frac{k_\nu\delta}{2}.
\]
Then
\[
\nu(x)= r\nu(y)< r(\omega(y)+\varepsilon)< r\omega(y)+\frac{2}{k_\nu}\varepsilon\le 1+k_\nu\delta,
\]
and \eqref{auxineq3} yields that $x\in S(B_\nu,\delta)$. Assume now that
\[
r\omega(y)> 1+\frac{k_\nu\delta}{2}.
\]
Then
\[
\nu(x)>r(\omega(y)-\varepsilon)\ge r\omega(y)-\frac{2\varepsilon}{k_\nu}>1+\frac{k_\nu\delta}{2}-\frac{k_\nu\delta}{2}=1
\]
and so $x\notin S(B_\nu,-\delta)$. This shows that we may decide weak membership in $B_\nu$ with a $\delta$-approximation to $\nu$. In fact we just need one oracle call to \textsc{approx}.

Let us use (ii) as an oracle and solve (i).   Let $x \in \mathbb{Q}^n$ where $\|x\|\in(1/2, 3/2)$ and a rational $\delta > 0$ be given. Again, observe that $\nu(x)\in [a_1,b_1]$, where $a_1= k_\nu/2$ and $b_1=3K_\nu/2$.  
Suppose that for an integer $i\ge 1$ we showed that $\nu(x)\in [a_i,b_i]$.
Let
\begin{equation}\label{defr}
r=\frac{a_i+b_i}{2},\qquad \varepsilon=\frac{b_i-a_i}{2K_\nu(b_i+a_i)},
\end{equation}
and consider $y=x/r$. Assume first that $y\in S(B_\nu,\varepsilon)$.  Then the right inclusion in \eqref{auxineq3} yields
$\nu(y)\le 1+K_\nu\varepsilon$ and thus
\[
\nu(x)=r\nu(y)\le \frac{a_i+b_i}{2}(1+K_\nu\varepsilon)= \frac{3}{4} b_i+\frac{1}{4}a_i.
\]
In this case we set $a_{i+1}=a_i$ and $b_{i+1}= 3b_i/4+ a_i/4$. Assume now that $y\notin S(B_\nu,-\varepsilon)$.  Then the left inclusion in \eqref{auxineq4} yields
\[
\nu(x)> r(1-K_\nu\varepsilon)= \frac{1}{4}b_i+\frac{3}{4}a_i.
\]
In this case we set  $a_{i+1}= b_i/4+3a_i/4$ and  $b_{i+1}=b_i$.

In either case, we obtain that $\nu(x)\in [a_{i+1},b_{i+1}]$.
Clearly, the sequence of intervals $\{[a_i,b_i] : i\in\mathbb{N} \}$ is nested and their successive lengths decrease by a factor of $3/4$.
\blue{Let $m$ be the smallest integer such that  
\[
m>1+\frac{\log_2 b_1 -\log_2 2\delta}{2-\log_2 3} \quad \text{if}\; 2\delta b^{-1}\le 1,
\] 
and otherwise set $m=1$.  Then}
\[b_m-a_m=
\left(\frac{3}{4}\right)^{m-1}(b_1-a_1)<\blue{\left(\frac{3}{4}\right)^{m-1}b_1}<2\delta.
\]
\blue{Clearly $m$ is polynomial, in fact linear, in}
$\langle K_\nu \rangle+\langle k_\nu \rangle+\langle \delta \rangle$.  
Setting $\omega(x)\coloneqq (a_m+b_m)/2$,  we obtain a $\delta$-approximation of $\nu(x)$.  This shows that we may determine a $\delta$-approximation to $\nu$ with $m$ oracle calls to  \textsc{wmem} in $B_\nu$.
\end{proof}

\begin{corollary}\label{cor:polduality}
A norm is polynomial-time approximable (resp.\ NP-hard to approximate) if and only if its dual norm is polynomial-time approximable (resp.\ NP-hard to approximate).
\end{corollary}

We end this section with a word about \emph{Mahler volume} \cite{BM87}. For any norm $\nu:\mathbb{R}^n\to [0,\infty)$, let 
$\operatorname{Vol}_n(B_{\nu})$ denote the volume of its unit ball $B_\nu$.  The Mahler volume of $\nu$ is defined as
\[
M(\nu)\coloneqq \operatorname{Vol}_n(B_{\nu})\operatorname{Vol}_n(B_{\nu^*}).
\]
A particularly nice property of the Mahler volume is that it is invariant under \emph{any} invertible linear transformation, regardless of whether it is volume-preserving or not.
 
\begin{corollary}
If the weak membership problem in $B_{\nu}$ is polynomial-time decidable, then $M(\nu)$ is polynomial-time approximable.
\end{corollary}
\begin{proof}
If the \textsc{wmem} in $B_\nu$ is polynomial-time decidable, then it follows from \cite{DFK91} that there exist polynomial-time algorithms to approximate $\operatorname{Vol}_n(B_{\nu})$ to any given error $\varepsilon>0$. By Corollary~\ref{cor:polduality}, the \textsc{wmem} in $B_{\nu^*}$ is also polynomial-time 
decidable and thus the same holds for  $\operatorname{Vol}_n(B_{\nu^*})$.
\end{proof}
Mahler volume is more commonly defined for a centrally symmetric compact convex set but as we mentioned before Corollary~\ref{cor:cscc}, this is equal to a 
unit norm ball for an appropriate choice of norm.

\section{Weak membership in dual cones}\label{sec:cones}

In this section, we move our discussion from balls to cones. While every ball is, by definition, a norm ball, a (proper) cone may not be a norm cone, 
i.e., of the form $\{ x\in \mathbb{R}^n : \| Ax \| \le c^\mathsf{T}x \}$ for some norm $\|\cdot\|$ and $A \in \mathbb{R}^{n \times n}$, $c \in \mathbb{R}^n$.  
So the results in this section would not in general follow from the previous sections.

Let $K\subset \mathbb{R}^n$ be a \emph{proper cone} in $\mathbb{R}^n$, i.e., $K$ is a closed convex pointed\footnote{By pointed, we mean that 
$K \cap (-K)=\{0\} $.} cone with non-empty interior. Then its \emph{dual cone},
\[
K^*\coloneqq \{x\in\mathbb{R}^n : y^\mathsf{T} x\ge 0 \;\text{for every}\; y\in K\},
\]
is also a proper cone \cite{Roc}. The main result of this section is an analogue of Theorem~\ref{polduality} for such cones: The weak membership problem for $K^*$ is polynomial-time reducible to the weak membership problem for $K$.

It is well-known that deciding \textsc{mem} for the cone of copositive matrices is NP-hard \cite{MK87}. This result  has recently been  extended \cite{DG14}: \textsc{wmem} in the cone of copositive matrices and \textsc{wmem} in its dual cone, the cone of completely positive matrices, are both NP-hard problems. Our result in this section generalizes this to arbitrary proper cones.

We first recall a well-known result regarding the interior points of $K^*$.
\begin{lemma}\label{interptK'}  Let $K\subseteq \mathbb{R}^n$ be a closed convex cone.  Let $b$ be an interior point of $K^*$, i.e., 
$b+z\in K^*$ for all $z\in B(0,\varepsilon_b)$ for some $\varepsilon_b >0$.  Then
\begin{equation} \label{interptK'1}  
b^\mathsf{T} x\ge \varepsilon_b \|x\|
\end{equation}
for every $x\in K$.
\end{lemma}
\begin{proof}  Let $x\in K\setminus\{0\}$.  Then $c\coloneqq b- \varepsilon_b x/ \|x\|\in K^*$.  Hence $c^\mathsf{T} x\ge 0$, which implies \eqref{interptK'1}.
\end{proof}

We now discuss the notion of \textsc{wmem} in $K$.  Recall that $x\in K\setminus\{0\}$ if and only if $tx\in K$ for each $t>0$.  Hence it \red{suffices} to define \textsc{wmem} in $K$ for \red{$x$
with Euclidean norm $\lVert x \rVert = 1$; but as such an $x$ may be not have rational coordinates, we instead} define a \textsc{wmem} problem for $x\in \mathbb{Q}^n$ \blue{that satisfies} $\frac{1}{2} <\|x\|<1$.
 
Let $a\in\mathbb{Q}^n$ and $b\in\mathbb{Q}^n$ be in the interior of $K$ and $K^*$ respectively. By Lemma~\ref{interptK'},
\begin{equation}\label{defPab}
P_b\coloneqq \{x\in K:b^\mathsf{T} x=1\}, \quad P^*_a=\{y\in K^*: a^\mathsf{T}y=1\}
\end{equation} 
are compact convex sets of dimension $n-1$.  Hence the sets 
\blue{$P_b-(b^\mathsf{T}a)^{-1}a$ and $P_a^*-(a^\mathsf{T}b)^{-1}b$} 
are full-dimensional compact convex sets in the orthogonal complements of $\operatorname{span}(b)$ and $\operatorname{span}(a)$ respectively.
In fact $P_b$ and $P_a^*$ are compact convex sets of maximal dimension in the affine hyperplanes 
\[
H_b\coloneqq \{z\in\mathbb{R}^n: b^\mathsf{T} z=1\}, \quad  H_a\coloneqq \{z\in\mathbb{R}^n: a^\mathsf{T} z=1\}
\]
respectively. We may also view $H_b$ and $H_a$ as the affine hulls of $P_b$ and $P_a^*$ respectively.

\red{As the cones $K$ and $K^*$ are noncompact, these hyperplane sections $P_b$ and $P_a^*$ serve as their compact proxies, allowing us to encode $K$ and $K^*$  (for a Turing machine).
We will assume knowledge of four positive rational numbers $\rho_a' < \rho_a$ and $\rho_b' < \rho_b$ such that
\begin{gather*}
B(0,\rho_a') \cap H_a \subseteq P_a^*-(a^\mathsf{T}b)^{-1}b \subseteq B(0,\rho_a) \cap H_a, \\
B(0,\rho_b') \cap H_b \subseteq P_b-(b^\mathsf{T}a)^{-1}a \subseteq B(0,\rho_b) \cap H_b.
\end{gather*}
$K$ will be encoded as $(n, a,b,\rho_a', \rho_a) \in \mathbb{Q}^{2n+3}$ and $K^*$ as  $(n, a,b,\rho_b',\rho_b) \in \mathbb{Q}^{2n+3}$. So
\[
\langle K \rangle\coloneqq \langle n \rangle+\langle a \rangle+\langle b \rangle + \langle \rho_a' \rangle+\langle \rho_a \rangle, \quad
\langle K^* \rangle\coloneqq \langle n \rangle+\langle a \rangle+\langle b \rangle + \langle \rho_b' \rangle+\langle \rho_b \rangle.
\]}%
While  the numbers $\rho_a, \red{\rho_a'}, \rho_b, \red{\rho_b'}$ do not appear explicitly in our proofs, they are needed implicitly when we invoke the Yudin--Nemirovski Theorem.

Given any $x\ne 0$, observe that $x\in K$ if and only if $x/(b^\mathsf{T} x) \in P_b$.
Thus the membership problem for $K$ is equivalent to the membership problem for $P_b$. We show in the following that this extends,  in an appropriate sense, to weak membership as well.
\begin{lemma}\label{eqWMEMconePb}
Let $x \in \mathbb{Q}^n$ with $1/2 <\|x\|<1$ and $b \in \mathbb{Q}^n$ with $b^\mathsf{T} x> 0$. Then the following problems are polynomial-time inter-reducible:
\begin{enumerate}[\upshape (i)]
\item  Decide weak membership of $x$ in $K$.
\item Decide weak membership of  $y\coloneqq  x/(b^\mathsf{T} x)$ in $P_b$ relative to $H_b$.
\end{enumerate}
\end{lemma}
\begin{proof}
Suppose that $0< \delta < b^\mathsf{T} x/(2\|b\|)$.   Let $z\in\mathbb{R}^n$ and $\|z\|\le \delta$. 
Clearly,
\[
b^\mathsf{T} (x+z)=b^\mathsf{T} x +b^\mathsf{T}z\ge b^\mathsf{T} x - \|b\|\|z\|\ge \frac{1}{2}b^\mathsf{T} x >0.
\]
In the following, we let $y\coloneqq  x/(b^\mathsf{T} x)$ and $u\coloneqq (x+z)/\bigl(b^\mathsf{T} (x+z)\bigr)\in H_b$.

Suppose that we can solve (i), i.e., for any rational $\delta>0$ and  $x\in \mathbb{Q}^n$ with $1/2<\|x\|<1$  we can decide whether $x\in S(K,\delta)$ or $x\notin S(K,-\delta)$.  Let $\varepsilon > 0$ be rational and choose $\delta$ rational so that
\[
\frac{(b^\mathsf{T} x)^2}{8\|b\|}\varepsilon < \delta < \frac{(b^\mathsf{T} x)^2}{4\|b\|}\varepsilon.
\]

Consider first the case $x\notin S(K,-\delta)$. There exists $z\in \mathbb{R}^n$, $\|z\|\le \delta$ such that $x+z\notin K$. So $u\notin P_b$.  Since
\begin{align*}
\red{y-u}&=\frac{1}{(b^\mathsf{T} x)(b^\mathsf{T} (x+z))}[(b^\mathsf{T} (x+z))x - (b^\mathsf{T} x)(x+z)]\\
&=\frac{1}{(b^\mathsf{T} x)(b^\mathsf{T} (x+z))}[(b^\mathsf{T} z)x - (b^\mathsf{T} x)z],
\end{align*}
we obtain
\[
\|y-u\|\le \frac{2}{(b^\mathsf{T} x)^2}(2\|b\|\|x\|\|z\|)\le  \frac{4\|b\|\delta}{(b^\mathsf{T} x)^2} < \varepsilon.
\]
Hence $y\notin S_{H_b}(P_b,-\varepsilon)$.

Consider now the case $x\in S(K,\delta)$. There exists $z\in \mathbb{R}^n$, $\|z\|\le \delta$ such that $x+z\in K$.
The same line of argument as above yields that  $\red{y}\in \blue{S_{H_b}(P_b,\varepsilon)}$. Together the two cases show that if we can decide \textsc{wmem} in $K$ with inputs $x$, $\delta$, then we can decide \textsc{wmem} in $P_b$ relative to $H_b$ with inputs $y$, $\varepsilon$.

Suppose we can solve (ii), i.e.,  for any rational $\varepsilon > 0$ and $x\in \mathbb{Q}^n$ with $1/2 <\|x\|<1$, $b^\mathsf{T} x>0$, we can decide whether $y\in \blue{S_{H_b}(P_b,\varepsilon)}$ or $y\notin S_{H_b}(P_b,-\varepsilon)$.  

Let $x\in \mathbb{Q}^n$ with $1/2 <\|x\|<1$. We start by excluding the trivial case when $b^\mathsf{T} x\le 0$.  By Lemma~\ref{interptK'},  $x\notin K$ and thus $x\notin S(K,-\delta)$ for any $\delta >0$. So we may assume henceforth that $b^\mathsf{T} x> 0$. Let $\delta > 0$ be rational  and set $\varepsilon\coloneqq \delta/(b^\mathsf{T} x)$.

Consider first the case $y\notin S_{H_b}(P_b,-\varepsilon)$.  There exists $v\in H_b\setminus{P_b}$ such that $\|v-y\|\le \varepsilon$.  Let $z=(b^\mathsf{T} x)(v-y)$.  So
\[
\|z\|\le (b^\mathsf{T} x)\varepsilon =\delta.
\]
Hence $(b^\mathsf{T} x)v=x+z\notin K$ and so $x\notin S(K,-\delta)$.  

Consider now the case $y\in \blue{S_{H_b}(P_b,\varepsilon)} $. The same line of argument as above yields  that  $x\in S(K,\delta)$.  Together the two cases show that if we can decide \textsc{wmem} in $P_b$ relative to $H_b$ with inputs $y$, $\varepsilon$, then we can decide \textsc{wmem} in $K$ with inputs $x$, $\delta$.
\end{proof}

Lemma~\ref{eqWMEMconePb} may be viewed as a compactification result: We transform a problem involving a noncompact object $K$ to a problem involving a compact object $P_b$. The motivation is so that we may apply the Yudin--Nemirovski Theorem later. 

\begin{theorem}\label{WMEMcone}
Let $K\subset \mathbb{R}^n$ be a proper cone and $K^*$ be its dual.  Let $a \in \mathbb{Q}^n$ and $b\in\mathbb{Q}^n$ be  interior points of $K$ and $K^*$ respectively that satisfy $b^\mathsf{T} a=1$.  Then the \textsc{wmem} problem for $K^*$ is polynomial-time reducible to the \textsc{wmem} problem for $K$.
\end{theorem}
\begin{proof} 
Note that such a pair of $a$ and $b$ must exist for any proper cone. Let $a,b\in \mathbb{Q}^n$ be interior points contained in balls of radii $\varepsilon_a$, $\varepsilon_b >0$ within $K$, $K^*$
respectively. So $b^\mathsf{T} a > 0$.  If $b^\mathsf{T} a=1$, we are done. Otherwise set $a' = a/(b^\mathsf{T} a)\in \mathbb{Q}^n$.  Then $b^\mathsf{T}a'=1$ and $a'$ is contained in a ball of radius
$\varepsilon_{a'}=\varepsilon_a/(b^\mathsf{T}a)$ within $K^*$.

By Lemma~\ref{eqWMEMconePb}, we just need to show that the \textsc{wmem} problem for $P_a^*$ relative to $H_a$ is polynomial-time reducible to the \textsc{wmem} problem for $P_b$ relative to $H_b$. Since $b^\mathsf{T} a=1$, $H_b -a =b^{\perp}$, the orthogonal complement of $b$, and can be identified with $\mathbb{R}^{n-1}$ by an orthogonal change of coordinates.  We set $K_b\coloneqq  P_b-a$, a compact closed set in $\mathbb{R}^{n-1}$ containing the origin $0\in \mathbb{R}^{n-1}$.  Moreover $B(0,\varepsilon_a)\subset K_b$, where $B(0,\varepsilon_a)$ here is an $(n-1)$-dimensional ball in $\mathbb{R}^{n-1}$. It is enough to show that the \textsc{wmem} problem for $P_a^*$ relative to $H_a$ is polynomial-time reducible to the \textsc{wmem} problem\footnote{When we refer to the \textsc{wmem} or \textsc{wval} problem for $K_b$, we mean its \textsc{wmem} or \textsc{wval} problem as a subset of $\mathbb{R}^{n-1}$.} for $K_b$. We would also need to invoke the fact that the \textsc{wval} problem for $K_b$ is polynomial-time reducible to the \textsc{wmem} problem for $K_b$ by the Yudin--Nemirovski Theorem. \red{The following sequence of polynomial-time reductions outlines the idea of our proof:
\begin{multline*}
\text{\textsc{wmem} for}\; K\Rightarrow \text{\textsc{wmem} for}\; P_b \;\text{relative to}\; H_b \\
\Rightarrow \text{\textsc{wmem} for}\; P_a^* \;\text{relative to}\; H_a \Rightarrow \text{\textsc{wmem} for}\; K^*.
\end{multline*}}%
Let $c\in \mathbb{Q}^n\cap H_a$.  Given a rational $\delta>0$ we need to decide whether $c\notin S_{H_a}(P_a^*,-\delta)$ or $c\in S_{H_a}(P_a^*,\delta)$.  Let $\varepsilon>0$ be rational with
\begin{equation}\label{condeps}
\varepsilon <\min\left\{\frac{1}{4(1+\|c\|)}, \frac{\delta}{4(1+\|c\|)(\|b-c\|)} \right\},
\end{equation}
where $\delta/0 \coloneqq \infty$ if $b = c$.  It follows from \eqref{condeps} that
\begin{equation}\label{condeps1}
\tau\coloneqq (1+\|c\|)\varepsilon\le \frac{1}{4},\qquad \left\|c-\frac{c+\tau b}{1+\tau}\right\| \le \delta,\qquad \left\|c - \frac{c-2\tau b}{(1-2\tau}\right\|\le \delta.
\end{equation}

Observe that $c$ defines a linear functional $b^{\perp} \to \mathbb{R}$, $x\mapsto c^\mathsf{T} x$.  Consider the \textsc{wval} problem for $K_b$ with $\gamma=-c^\mathsf{T} a$:
Either $c^\mathsf{T}x\ge -c^\mathsf{T} a -\varepsilon$ for all $x\in \red{S_{H_b}(K_b,-\varepsilon)}$ or  $c^\mathsf{T}x\le -c^\mathsf{T}a +\varepsilon$ for some $x\in \red{S_{H_b}(K_b,-\varepsilon)}$.
We will show that in the first case $c\in \blue{S_{H_a}(P_a^*,\delta)}$ and in the second case $c\not\in S_{H_a}(P_a^*,-\delta)$ for a corresponding $\delta >0$.

Consider first the case $c^\mathsf{T}x\ge -c^\mathsf{T} a -\varepsilon$ for all $x\in \red{S_{H_b}(K_b,-\varepsilon)}$, or, equivalently, $c^\mathsf{T}y \ge -\varepsilon$ for all $y = x + a \in S_{H_b}(P_b,-\varepsilon)$. 
We claim that $c^\mathsf{T}y\ge -(1+\|c\|)\varepsilon$ for all $y\in P_b$. This holds for $y \in  \red{S_{H_b}(P_b,-\varepsilon)}$ since  $c^\mathsf{T}y \ge -\varepsilon \ge -(1+\|c\|)\varepsilon$. For 
$y\in P_b\setminus \red{S_{H_b}(P_b,-\varepsilon)}$, there exists $x\in \red{S_{H_b}(P_b,-\varepsilon)}$ such that $\|y-x\|\le \varepsilon$.
Thus $c^\mathsf{T} y=c^\mathsf{T}x+ c^\mathsf{T}(y-x)\ge -\varepsilon -\|c\|\|y-x\|=-(1+\|c\|)\varepsilon$.
Then for any $ y\in P_b$,
\[
\frac{1}{1+\tau}(c+\tau b)^\mathsf{T} y\ge 0 \quad \Rightarrow\quad \frac{1}{1+\tau}(c+\tau b)\in P_a^*.
\] 
By the middle inequality in \eqref{condeps1}, we obtain $c\in \blue{S_{H_a}(P_a^*,\delta)}$.

Consider now the case $c^\mathsf{T}x\le -c^\mathsf{T}a +\varepsilon$ for some $x\in \blue{S_{H_b}(K_b,\varepsilon)}$, or, equivalently, $c^\mathsf{T} y \le \varepsilon$ for some $y = x + a\in \blue{S_{H_b}(P_b,\varepsilon)}$.  Hence there exists $z\in P_b$ such that $\|z - y\|\le \varepsilon$ and so $c^\mathsf{T} z=c^\mathsf{T} y +c^\mathsf{T}(z- y)\le (1+\|c\|)\varepsilon=\tau<1/4$ by the left inequality in \eqref{condeps1}.  Then
\[\frac{1}{1-2 \tau}(c-2\tau b)^\mathsf{T} z\le - \tau \quad \Rightarrow \quad \frac{1}{1-2 \tau}(c-2\tau b) \not\in P_a^*.\] 
By the right inequality in \eqref{condeps1}, we obtain $c\not\in S_{H_a}(P_a^*,-\delta)$.
\end{proof}

\section{Approximation of Fenchel duals}\label{sec:fenchel}

Let $C\subseteq \mathbb{R}^n$ and $f: C \to \mathbb{R}$. Since the epigraph of $f$,
$\operatorname{epi}(f)=\{(x,t)\in C \times \mathbb{R} : f(x)\le t\}$, is in general noncompact, we introduce the following variant that preserves all essential features of the epigraph but has the added advantage of facilitating complexity theoretic discussions. For any $\alpha\in\mathbb{R}$, we let
\[
\operatorname{epi}_\alpha(f)=\{(x,t)\in C \times (-\infty, \alpha] :  f(x)\le t\}
\]
and call this the \emph{$\alpha$-epigraph} of $f$.
Clearly $f$ is a convex function if and only if $\operatorname{epi}_\alpha(f)$ is a convex set for all $\alpha \in \mathbb{R}$.

\begin{definition}\label{defncompfepi}  Let $C\subseteq \mathbb{R}^n$ be a bounded set with nonempty interior.  Let $f: C\to\mathbb{R}$ be a bounded function.  We define the following \emph{approximation problems} (\textsc{approx}).
\begin{enumerate}[\upshape (i)]
\item Approximation problem for $f$: Given any $x\in \mathbb{Q}^n\cap C$ and any rational $\varepsilon >0$, find an $\omega(x)$ 
such that \blue{$ f(x)-\varepsilon<\omega(x)<f(x)+\varepsilon$}.
\item  Approximation problem for $\mu\coloneqq \inf_{x\in C} f(x)$: Given any rational $\varepsilon>0$, find $\mu(\varepsilon)\in\mathbb{Q}$ such $\mu-\varepsilon<\mu(\varepsilon)<\mu+\varepsilon$. 
\end{enumerate}
\end{definition}
(i) is of course a generalization of Definition~\ref{def:approx} from norms to a more general function. We will show that (i) and (ii) are polynomial-time inter-reducible. For this purpose, we will need a useful corollary \cite[Corollary~4.3.12]{GLS88} of the Yudin--Nemirovski Theorem (cf.\ Theorem~\ref{thm:YN})
with the \textsc{wopt} problem in place of the \textsc{wval} problem.

\begin{corollary}[Yudin--Nemirovski]\label{cor:YN}
Let $C \subseteq \mathbb{R}^n$ be a compact convex set with nonempty interior for which we have knowledge of $a \in \mathbb{Q}^n$ and $0< r \le R \in \mathbb{Q}$ such that $B(a,r) \subseteq C \subseteq B(\blue{a},R)$. Then the \textsc{wopt} problem for $C$ is polynomial-time reducible to the \textsc{wmem} problem for $C$.
\end{corollary}

We will rely on this to show that for a convex function $f : C \to \mathbb{R}$, the approximation problem for $\inf_{x \in C} f(x)$ is polynomial-time reducible to the  approximation problem for $f$.

\begin{lemma}\label{mincomput}
Let $C\subseteq \mathbb{R}^n$ be a compact convex set with nonempty interior where \textsc{mem}  in $C$ can be checked in polynomial time. Let $f:C \to \mathbb{R}$ be a continuous convex functions with $\lvert f(x) \rvert \le \alpha$ for some rational $\alpha > 0$.
Suppose that there exists a rational $\delta>0$ such that  
\begin{equation}\label{mindeltacond}
\mu\coloneqq \min_{x\in C} f(x)=\min_{x\in S(C,-\delta)} f(x) .
\end{equation}
Then the approximation problem for $\mu$ is polynomial-time reducible to the approximation problem for $f$.
\end{lemma}
Note that we require knowledge of the values of both $\alpha$ and $\delta$, not just of their existence.
\blue{We need the condition \eqref{mindeltacond} to ensure that no minimizer of $f$ lies on the boundary of $C$ and that any minimizer is at least distance $\delta$ away from the boundary.}%
\begin{proof}[of Lemma~\ref{mincomput}]
We will show that \textsc{wopt} in $\operatorname{epi}_{2\alpha}(f)$ yields a solution to \textsc{approx} for $\mu$. The result then follows from two polynomial-time reductions: \textsc{wopt} in $\operatorname{epi}_{2\alpha}(f)$ can be reduced to \textsc{wmem} in $\operatorname{epi}_{2\alpha}(f)$, \textsc{wmem} in $\operatorname{epi}_{2\alpha}(f)$ can be reduced to \textsc{approx} for $f$.

As $f$ is a continuous convex function and $C$ is compact with nonempty interior, $C' \coloneqq \operatorname{epi}_{2\alpha}(f)$ is a compact convex set with interior in $\mathbb{R}^{n+1}$.
We claim that the \textsc{wmem} in $C'$ is polynomial-time reducible to the approximation problem for $f$.  Let $\varepsilon\in \mathbb{Q}$ with $0 < \varepsilon < \alpha$  and $(x,t)\in\mathbb{Q}^{n+1}$. If $x\notin C$  or $t>2\alpha$, then $(x,t)\notin C'$ and so $(x,t)\notin S(C',-\varepsilon)$. Now suppose $x\in C$ and $t\le 2\alpha$. An oracle call to the approximation problem for $f$ gives us $\omega(x)$ with $\omega(x)-\varepsilon<f(x)<\omega(x)+\varepsilon$. If $t\ge \omega(x)$, then as $(x,t)+(0,\varepsilon)\in C'$, it follows 
that $(x,t)\in S(C',\varepsilon)$.  If $t<\omega(x)$, then as $(x,t)-(0,\varepsilon)\notin C'$, it follows that $(x,t)\notin S(C',-\varepsilon)$.

By Corollary~\ref{cor:YN}, \textsc{wopt} in $C'$ is polynomial-time reducible to \textsc{wmem} in $C'$.  Therefore given $\varepsilon \in \mathbb{Q}$ with $0 < \varepsilon < \min(\alpha,\delta)$ and $\gamma=(0,\dots,0,-1)\in \mathbb{Z}^{n+1}$, by an oracle call to \textsc{wmem} in $C'$,
we may find $(y,s)\in S(C',\varepsilon)$ such that 
\[
\gamma^\mathsf{T}(x,t)=-t\le \gamma^\mathsf{T}(y,s)+\varepsilon=-s+\varepsilon
\]
for all $(x,t)\in S(C',-\varepsilon)$. We claim that $s = \mu(\varepsilon)$, the required approximation to $\mu$.
Since \red{$\varepsilon< \delta$}, it follows that $S(C',-\varepsilon)\supseteq S(C',-\delta)$.  The assumption \eqref{mindeltacond} ensures that $(x^\star, \mu)\in S(C,-\delta)$ where $f(x^\star) = \mu$.
Hence we deduce that $s\le \mu+\varepsilon$, i.e., $\mu\ge s-\varepsilon$.   As $(y,s)\in S(C',\varepsilon)$, it follow that there exists $(x',t')\in C'$ such that 
$t'\ge f(x')$ and $\lvert t'-s\rvert \le \varepsilon$. So $s\ge t'-\varepsilon\ge \mu-\varepsilon$. Thus
$\mu-\varepsilon \le s \le \mu+\varepsilon$, but starting with $2\varepsilon$ in place of $\varepsilon$ allows us to replace `$\le$' by `$<$' as required by Definition~\ref{defncompfepi}\red{(ii)}.
\end{proof}
 
We now turn to the computational complexity of \emph{Fenchel dual} \cite{AM, Roc}. Our results here require that $f$ be defined on all of $\mathbb{R}^n$.  Recall that for a function $f : \mathbb{R}^n \to \mathbb{R}$, its Fenchel dual is defined to be the function $f^*: \mathbb{R}^n\to (-\infty,\infty]$,
\[
f^*(y)\coloneqq \sup_{x\in \mathbb{R}^n} \blue{\bigl( y^\mathsf{T} x -f (x)\bigr)}.
\]
The Fenchel dual is also known as the \emph{Fenchel conjugate} and the map $f \mapsto f^*$ is sometimes  called the \emph{Legendre transform}. It is well-known that $f^*$ is always a convex function, being the pointwise supremum of a family of affine functions  $y \mapsto y^\mathsf{T}x -f(x)$.  It is  also well-known that $f$ is \red{a lower semicontinuous proper convex function if and only if} $f^{**}=f$.

Suppose that given any inputs $x \in \mathbb{Q}^n$ and $0 < \varepsilon \in \mathbb{Q}$, we can compute $f(x)$ to within precision $\varepsilon$  in polynomial-time. What can be we say about the complexity of computing $f^*(y)$ for an input $y \in \mathbb{Q}^n$ to a certain precision?  We will see that if $f$ is not convex, then the computation of $f^*$ can be NP-hard at least for some $y$. However, when $f$ is convex and satisfies certain growth conditions, computing $f^*$ is a problem that is polynomial-time reducible to computing $f$. Furthermore $f^*$ would satisfy the same growth conditions so that computing $f$ and computing $f^*$ are in fact polynomial-time inter-reducible.

Let $g : \mathbb{R}^n \times \mathbb{R}^n \times \mathbb{R}^n \to \mathbb{R}$, $(x,y,z) \mapsto \sum_{i,j,k=1}^n a_{ijk} x_i y_j z_k$ be a multilinear function. Let $D=\{(x,y,z) \in \mathbb{R}^{3n}: \|x\|\le 1,\; \|y \| \le 1, \; \|z \| \le 1\}$. We define a nonconvex function $f$ as follows:  For $(x,y,z)\in D$, $f(x,y,z)\coloneqq -g(x,y,z)$.  For $(x,y,z)\notin D$, let $t=1/\max(\|x\|,\|y\|,\|z\|)$ and set $f(x,y,z)\coloneqq -g(tx,ty,tz)$. It is trivial to compute $f$ for any $(x,y,z) \in \mathbb{R}^{3n}$ but $f^*(0)=\max_{(x,y,z)\in D} g(x,y,z)$ is NP-hard to approximate in general \cite[Theorem~10.2]{HL13}.

In what follows let $f : \mathbb{R}^n \to \mathbb{R}$ be a  continuous convex function. We will assume that $f$ satisfies the following growth condition:
\begin{equation}\label{growthcondf1}
k_f\|x\|^s\le f(x)\le K_f\|x\|^t \quad \text{whenever} \quad \|x\|\ge r.
\end{equation} 
for some constants $0<k_f \le K_f$,  $1<s \le t$, and $r > 0$ depending on $f$.
We now show that $f^*$ must satisfy similar growth conditions
\begin{equation}\label{growthcondf*}
k_{f^*}\|y\|^{s'}\le f^*(y)\le K_{f^*}\|y\|^{t'}  \quad  \text{whenever} \quad \|y\|\ge r',
\end{equation}
but with possibly different constants.  
\begin{lemma}\label{equivgrowthcond}
Let $f:\mathbb{R}^n\to \mathbb{R}$ be a convex function and let $f^*:\mathbb{R}^n\to (-\infty,\infty]$ be its Fenchel dual.
Then $f$ satisfies \eqref{growthcondf1} if and only if $f^*$ satisfies \eqref{growthcondf*}.
\end{lemma}
\begin{proof}
For $\|x\|\ge r$, the lower bound in  \eqref{growthcondf1} and $y^\mathsf{T}x\le \|y\|\|x\|$ give
\begin{equation}\label{basineq}
y^\mathsf{T}x - f(x)\le \|y\|\|x\|-k_f\|x\|^s=\|x\|(\|y\| - k_f\|x\|^{s-1}).
\end{equation}
Observe that for $z \in [0, \infty)$, the maximum of $h(z)\coloneqq \|y\| z- k_f z^s$ is attained at
\[
z^\star=\left(\frac{\|y\|}{k_f s}\right)^{1/(s-1)},
\]
with maximum value
\[
h(z^\star)=\frac{s-1}{s}\|y\|z^\star=\frac{s-1}{s (k_f s)^{1/(s-1)}} \|y\|^{s/(s-1)}.
\]

Let $\mu\coloneqq \min_{\lVert x \rVert \le r} f(x)$.  Then
\[
\max_{\lVert x \rVert \le r} \blue{\bigl( y^\mathsf{T} x -f (x)\bigr)}\le \|y\| r-\mu.
\]
Combine this with \eqref{basineq} and we obtain
\[
f^*(y)\le \max\left(\|y\| r-\mu, \frac{s-1}{s (k_f s)^{1/(s-1)}} \|y\|^{s/(s-1)}\right).
\]
This last inequality yields the upper bound in \eqref{growthcondf*}
with 
\[
K_{f^*}= \frac{s-1}{s (k_f s)^{1/(s-1)}},\qquad t'=\frac{s}{s-1},\qquad r'\ge r_1,
\]
for a corresponding $r_1$  that depends on $k_f,s,r,\mu$.  More precisely, either $r_1=0$ or $r_1$ is the unique positive solution of
\[
r_1r-\mu=\frac{s-1}{s (k_f s)^{1/(s-1)}} r_1^{s/(s-1)}.
\]

To deduce the lower bound in \eqref{growthcondf*}, let $y$ be such that
\[
\|y\| \ge r^{t-1} K_f t.
\]
Choose $x = cy$ such that
\[
\|x\|= \left( \frac{\|y\|}{K_f t} \right)^{1/(t-1)}.
\]
It follows that $\| x \| \ge r $ and so  the upper bound in \eqref{growthcondf1} yields $f^*(y)\ge \|y\|\|x\|-K_f\|x\|^t $.
Hence we have the lower bound in \eqref{growthcondf*} with
\[
k_{f^*} = \frac{t-1}{t (K_f t)^{1/(t-1)}}, \qquad s'=\frac{t}{t-1}. 
\]
\end{proof}

\begin{theorem}\label{compcpmplexf}
Let  $f:\mathbb{R}^n\to  \mathbb{R}$ be a convex function satisfying  
\eqref{growthcondf1}.  Then the approximation problem for $f^*$ is polynomial-time reducible to the approximation problem for $f$.
\end{theorem}
\begin{proof}
We will compute an approximation of $f^*(y)$ with oracle calls to approximations of $f(x)$.

Suppose first that $y=0$ and we need to compute an approximation of $f^*(0)=\red{\sup}_{x\in\mathbb{R}^n} -f(x)$.  By the lower bound in \eqref{growthcondf1}, there is some $\rho_0=\rho(r, k_f,s)\in \mathbb{Q}\cap(0,\infty)$ such that $-f(x)<-f(0)$ whenever $\|x\|\ge \rho_0$.   Hence 
\[
f^*(0)=\max_{\|x\|\le \rho_0}-f(x)=-\min_{\|x\|\le \red{\rho_0}} f(x)=\blue{-}\min_{\|x\|\le \rho_0+1} f(x).
\]
Let $C=B(0,\rho_0+1)$. Since \textsc{mem} in a Euclidean ball $B(0,\rho)$ is clearly polynomial-time decidable, the conditions of Lemma~\ref{mincomput} are satisfied.  Hence \textsc{approx} for $f^*(0)$ is polynomial-time reducible to \textsc{approx} for $f$.

Suppose now that $y\ne 0$.  Clearly $f^*(y)\ge -f(0)$.  Let $\rho>r$, where $r$ is as in \eqref{growthcondf1}.  Let $f^*_\rho(y)\coloneqq \max_{\|x\|=\rho} \blue{\bigl( y^\mathsf{T}x -f(x)\bigr)}$.
As $y^\mathsf{T}x\le \|y\| \|x\|$, the lower bound in \eqref{growthcondf1} gives
\[
f^*_{\rho}(y)\le \|x\|(\|y\|-\blue{k_f}\|x\|^{s-1})= \rho(\|y\|-\blue{k_f}\rho^{s-1}).
\]
Hence there exists $\rho_1  = \rho(\|y\|,k_f,s)\in \mathbb{Q}\cap (r,\infty)$ such that \red{$-f(0)> f^*_{\rho}(y)$ for all $y \in \mathbb{R}^n$  whenever $\rho\geq \rho_1$}.  Therefore
\[f^*(y)=-\min_{\|x\|\le \rho_1}\blue{\bigl( f(x)-y^\mathsf{T}x\bigr)}=-\min_{\|x\|\le \rho_1+1} \blue{\bigl(f(x)-y^\mathsf{T}x\bigr)}.
\]
Let $0 \ne y\in \mathbb{Q}^n$ and $C=B(0,\rho_1+1)$.  Then the conditions of Lemma~\ref{mincomput} are satisfied.  Hence \textsc{approx} for $f^*(y)$ is polynomial-time reducible to \textsc{approx} for $f$.
\end{proof}

Since $f^{**} = f$ for a convex function and by Lemma~\ref{growthcondf1}, $f$ and $f^*$ both satisfy the polynomial growth condition if either one does, we obtain the following.
\begin{corollary}
Let  $f:\mathbb{R}^n\to  \mathbb{R}$ be a convex function satisfying  
\eqref{growthcondf1}.  The approximation problem for $f^*$ is polynomial-time computable (resp.\ NP-hard) if and only if the approximation problem for $f$ is polynomial-time computable (resp.\ NP-hard).
\end{corollary}

\section{\red{Conclusion}}

In this article, we have focused on establishing equivalence in the computational complexity of dual objects for several common convex objects and common notions of duality. These results are expected to have immediate  applications in many areas. We conclude our article with two such examples.

Drawing from our own work, we rely on the results in Sections~\ref{sec:norm} and \ref{sec:approx} to deduce that the nuclear norm for higher-order tensors is NP-hard to compute \cite[Corollary~8.8]{FL} and likewise for the dual norm of an operator $(p,q)$-norm when $1 \le q < p \le \infty$ or when $p=q \notin \{ 1,2,\infty\}$ \cite[Section~7]{FL}.

Following the notations in \cite{JLZ}, we let $\smash{\Sigma^2_{\nabla^2_{n,4}}}$\!\!\!\! denote the cone of \textsc{sos}-convex quartic forms \cite{HN} and  $\Sigma^2_{n,4} \cap \mathbb{S}^{n^4}_{\operatorname{cvx}}$ denote the cone of convex quartic forms that are \textsc{sos}. Using the results in Section~\ref{sec:cones} and \cite[Proposition~5.1 and Theorem~5.4]{JLZ}, we easily deduce that membership in the dual cone of $\smash{\Sigma^2_{\nabla^2_{n,4}}}$\!\!\!\! is polynomial-time whereas membership in the dual cone of $\Sigma^2_{n,4} \cap \mathbb{S}^{n^4}_{\operatorname{cvx}}$ is NP-hard --- observations that are new to the best of our knowledge. Furthermore, if we assume that $\mathit{P} \ne \mathit{NP}$, then it follows that the containment of  $\smash{\Sigma^2_{\nabla^2_{n,4}}}$\!\!\!\! in  $\Sigma^2_{n,4} \cap \mathbb{S}^{n^4}_{\operatorname{cvx}}$ is strict, verifying \cite[Theorem~4.1]{JLZ}.

\section*{Acknowledgment}

We are very grateful to the two anonymous referees for their exceptionally helpful suggestions and comments.
We would like to thank Lev~Reyzin for telling us about the various variants of the membership problem, and to Shuzhong Zhang for informing us that the problem of complexity of dual cones is still \blue{open} and pointing us to \cite{DG14, JLZ}.

\bibliographystyle{plain}

\begin{thebibliography}{1}
\bibitem{AM} S.~Artstein-Avidan and V.~Milman, ``The concept of duality in convex analysis, and the characterization of the Legendre transform,'' \emph{Ann.\ Math.}, \textbf{169} (2009), no.~2, pp.~661--674.


\bibitem{BM87} J.~Bourgain and V.~Milman, ``New volume ratio properties for convex symmetric bodies in $\mathbb{R}^n$,'' \emph{Invent.\ Math.}, \textbf{88} (1987), no.~2, pp.~319--340.

\bibitem{DG14} P.~J.~C.~Dickinson and L.~Gijben, ``On the computational complexity of membership problems for the completely positive cone and its dual,'' \emph{Comput.\ Optim.\ Appl.}, \textbf{57} (2014), no.~2, pp.~403--415.

\bibitem{DFK91} M.~Dyer, A.~Frieze, and R.~Kannan, ``A random polynomial-time algorithm for approximating the volume of convex bodies,'' \emph{J.\ Assoc.\ Comput.\ Mach.}, \textbf{38} (1991), no.~1, pp.~1--17. 

\bibitem{FL} S.~Friedland and L.-H.~Lim, ``Nuclear norm of higher-order tensors,'' \emph{preprint}, (2016), \url{http://arxiv.org/abs/1410.6072}.

\bibitem{GLS88} M.~Gr\"{o}tschel, L.~Lov\'{a}sz, and A.~Schrijver, \emph{Geometric Algorithms and Combinatorial Optimization}, 2nd Ed., Algorithms and Combinatorics, \textbf{2}, Springer-Verlag, Berlin, 1993.

\bibitem{Gu02}  L.~Gurvits, ``Classical deterministic complexity of Edmonds problem and quantum entanglement,'' \emph{Proc.\ ACM Symp.\ Theory Comput.} (STOC), \textbf{35}, pp.~10--19, ACM Press,  New York, NY, 2003.

\bibitem{HN} J.~W.~Helton and J.~Nie, ``Semidefinite representation of convex sets,'' \textit{Math.\ Program.}, \textbf{122} (2010), no.~1, pp.~21--64.

\bibitem{HL13} C.~J. Hillar and L.-H. Lim, ``Most tensor problems are NP-hard,''  \emph{J.\ Assoc.\ Comput.\ Mach.}, \textbf{60} (2013), no.~6, Art.~45, 39 pp.

\bibitem{JLZ} B.~Jiang, Z.~Li, and S.~Zhang, ``On cones of nonnegative quartic forms,'' \emph{Found.\ Comput.\ Math.}, (2016), to appear.

\bibitem{MK87}K.~G.~Murty and S.~N.~Kabadi, ``Some NP-complete problems in quadratic and nonlinear programming,'' \emph{Math.\ Programming}, \textbf{39} (1987), no.~2, pp.~117--129.

\bibitem{Roc}  R.~T.~Rockafellar, \emph{Convex Analysis},  Princeton Mathematical Series, \textbf{28}, Princeton University Press, Princeton, NJ, 1970.

\red{\bibitem{YN76} D.~B.~Yudin and A.~S.~Nemirovski, ``Informational complexity and
efficient methods for the solution of convex extremal problems,'' \emph{Matekon}, \textbf{13} (1977), no.~3, pp.~25--45.}


 \end{thebibliography}

\end{document}